\documentclass[12,reqno]{amsart}
\usepackage{fullpage}
\usepackage{amsfonts}
\usepackage[active]{srcltx}

\usepackage{amsmath,amssymb,amsthm,bm}
\usepackage{mathtools}
\usepackage{shuffle}
\usepackage{bm} 
\usepackage{enumitem}
\usepackage{amsrefs}
\usepackage{empheq}
\usepackage{wasysym}
\usepackage{verbatim}
\usepackage{graphicx}
\usepackage[
bookmarks=true,         
bookmarksnumbered=true, 
colorlinks=true, pdfstartview=FitV, linkcolor=blue, citecolor=blue,
urlcolor=blue]{hyperref}
\usepackage{microtype}
\theoremstyle{plain}
\newtheorem{theorem}{Theorem}[section]

\newtheorem{corollary}[theorem]{Corollary}

\newtheorem{lemma}[theorem]{Lemma}
\newtheorem{problem}[theorem]{Problem}
\newtheorem{proposition}[theorem]{Proposition}

\theoremstyle{remark}

\theoremstyle{definition}
\newtheorem{definition}[theorem]{Definition}
\newtheorem{hypothesis}[theorem]{Hypothesis}

\newtheorem{remark}[theorem]{Remark} 

\newcommand{\EE}{\mathbb{E}}

\newcommand{\R}{\mathbb{R}}

\newcommand{\RR}{\mathbb{R} }

\newcommand{\ot}{[0,t]}
\newcommand{\ott}{[0,T]}

\newcommand{\1}{{\bf 1}}

\newcommand{\be}{\beta}

\newcommand{\bp}{\mathbf{P}}

\newcommand{\cf}{\mathcal F}

\newcommand{\HH}{\mathfrak H}
\newcommand{\al}{\alpha}

\newcommand{\ga}{\gamma}

\newcommand{\si}{\sigma}

\newcommand{\vp}{\varphi}

\newcommand{\lt}{\left }
\newcommand{\rt}{\right}

\newcommand{\lc}{\left[}
\newcommand{\rc}{\right]}

\allowdisplaybreaks[4] 
  
\let\Section=\section
\def\section{\setcounter{equation}{0}\Section}

\def\RR{\mathbb{R} }

\def\EE{\mathbb{E}}

\def\de{{\delta}}

\def\si{{\sigma}}

\def\Ga{{\Gamma}}

\title{ Joint H\"older continuity of parabolic Anderson model
 }
\author[Y. Hu]{Yaozhong Hu}
\thanks{Y. Hu is supported by a startup fund from   University of Alberta at Edmonton.}
\address{Department of Mathematical and Statistical Sciences \\
University of Alberta at Edmonton \\
Edmonton, Alberta, T6G 2G1, Canada}
\email{yaozhong@ualberta.ca}
\author[K. L\^e]{Khoa  L\^e}
\address{Department of Mathematics\\
 South Kensington Campus\\
 Imperial College London\\ 
 London, SW7 2AZ, United Kingdom}
\email{n.le@imperial.ac.uk}

\begin{document}
\begin{abstract}  
	We show that the random field solution to the parabolic Anderson equation $(\partial_t-\frac12 \Delta)u=u\diamond \dot{W}$ is jointly H\"older continuous in space and time.
\end{abstract}
\maketitle

\setlength{\parindent}{1.5em}



\section{Introduction}
In this paper we  study the H\"older continuity 
and joint H\"older continuity of the solution to the following multiplicative stochastic heat equation:
\begin{equation}\label{spde}
\frac{\partial u(t,x)}{\partial t}=\frac{1}{2}\Delta u(t,x)+u(t,x)\,  \dot W(t,x)\,,
\end{equation} 
where the unknown $u=\left\{u(t,x)\,, t\ge 0, x\in \RR^d\right\}$ is a random field, $\Delta =\sum_{i=1}^d \frac{\partial^2}{\partial x_i^2}$ is the Laplacian,   and  $\dot W$ is general mean zero Gaussian noise whose covariance is given by
\[
\EE(\dot W(s,x)\dot W(t,y)=\ga_0(t-s)\ga(x-y)\,, \quad \forall \ 
s, t\ge 0\,, x, y\in \RR^d\,. 
\]
Here we assume that $\ga_0$ is a locally integrable function and $\ga(x)=\int_{\RR^d} e^{\iota x\cdot \xi} \mu(\xi)d \xi $ 
for some non-negative spectral density $\mu$ which will be specified in the next section. 
  The multiplication between $u $ and $ \dot W$ is the Wick product, whose precise meaning  is given in the following section by using the Skorohod integral.  
  
There are several papers on the H\"older continuity of the solution of stochastic partial differential equation. When the space dimension is $1$ and when  the noise $\dot W$ is space time white it is well-known that the solution is H\"older continuous of exponent 
$\be/2$ in time variable and $\be$ in space variable $x$ for any $\be<1/2$ (see \cite{chendalang} and \cite{sanzsarra} for some  work and for references). 

When the noise is a general Gaussian noise, there are also some work on 
H\"older continuity  in   \cite{HHLNT},  \cite{HHNT}, and \cite{hunualartsong}.
However, the H\"older continuity obtained in those papers are not sharp: 
when the noise is reduced to the one dimensional space-time white noise, one cannot obtain the known H\"older exponents. This paper contains three main contributions. First, we can allow the Gaussian noise $\dot W$ to be quite general, including the noise which is rough in space studied in recent work 
\cite{HHLNT},   \cite{HHLNT1}, and \cite{HHNT}. Secondly, our results are sharp, they yield optimal H\"older exponents when the noise is reduced to one dimensional space time white  noise.
Third, we obtain joint H\"older continuity in the sense of \cite{hule}.  

After the completion of the current paper, we have learnt that Balan et.al. in \cite{BSS} have obtained results along the same line. Although the assumptions on the spatial covariance of the noise are different, the method of \cite{BSS} also relies on Wiener chaos expansion. However, the current paper proceed further by considering joint H\"older continuity.

\section{Preliminary and notations} 
\def\cT{\mathcal{T}}  
Let us start by introducing our basic notation on Fourier transforms
of functions. The space of   Schwartz functions  on $\RR^d$ is
denoted by $\mathcal{S}=\mathcal{S}(\RR^d)$. Its dual, the space of tempered distributions, is $\mathcal{S}'$.  The Fourier
transform of a function $u \in \mathcal{S}$ or distribution $u\in \mathcal{S}'$  is defined with the normalization
\[ \mathcal{F}u ( \xi)  = \int_{\mathbb{R}^d} e^{- \iota
   \xi \cdot x } u ( x) d x, \]
so that the inverse Fourier transform is given by $\mathcal{F}^{- 1} u ( \xi)
= ( 2 \pi)^{- 1} \mathcal{F}u ( - \xi)$.

Let  $\ga_0(s)\,, 0< s <\infty$ be a real valued positive locally integrable function such that $\ga_0(s-t)$ is a  positive definite  function
of $s, t\in (0, \infty)$.
Let $\mu(\xi)$ be a non-negative function on $\RR^d$.

 Let $ \mathcal{D}((0,\infty)\times \RR^d)$ denote the space  of real-valued infinitely differentiable functions with compact support on $(0, \infty) \times \R^d$  and let  $\left\{W(\varphi), \varphi\in \mathcal{D}((0,\infty)\times \R^d)\right\}$   be a zero-mean Gaussian family defined on a complete probability space
$(\Omega,\cf,\bp)$, whose covariance structure
is given by
\begin{equation}\label{eq:cov1}
\EE\lc W(\vp) \, W(\psi) \rc
=  \int_{\R_{+}^2 \times\R^d}
\cf\varphi(s,\xi) \, \overline{\cf\psi(t,\xi)} \, \ga_0(s-t) dsdt \mu(\xi)d \xi 
\end{equation}
where the Fourier transforms $\cf\varphi,\cf\psi$ are understood as Fourier transforms in spatial variables only.

We  
denote by $\HH $ the  Hilbert space obtained by  completion of $ \mathcal{D}((0,\infty)\times \R)$ with respect to the inner product
\begin{equation}\label{eq: H_0 element H prod}
  \langle\varphi, \psi \rangle_{ \HH}=\int_{\R_{+}^2 \times\R^d}
\cf\varphi(s,\xi) \, \overline{\cf\psi(t,\xi)} \, \ga_0(s-t) dsdt \mu(\xi)d \xi \,.
  \end{equation}

We shall use Malliavin calculus to define the stochastic integral  
(\cite{hubook} and \cite{Nua}).  
If   a
random variables $F$ is  of the smooth and cylindrical form 
\begin{equation*}
F=f(W(\phi_1),\dots,W(\phi_n))\,,
\end{equation*}
with $\phi_i \in \HH$, $f \in C^{\infty}_p (\R^n)$ ($f$ is smooth and $f$ itself and   all
its partial derivatives have polynomial growth) then its    Malliavin  derivative  is the
$\HH$-valued random variable defined by
\begin{equation*}
DF=\sum_{j=1}^n\frac{\partial f}{\partial
x_j}(W(\phi_1),\dots,W(\phi_n))\phi_j\,.
\end{equation*}
The operator $D$ is closable from $L^2(\Omega)$ into $L^2(\Omega;
\HH)$  and we define the Sobolev space $\mathbb{D}^{1,2}$ as
the closure of the space of smooth and cylindrical random variables
under the norm
\[
\|F\|_{1,2}=\sqrt{\EE [F^2]+\EE [\|DF\|^2_{\HH}  ]}\,.
\]
If  $u \in L^2(\Omega;
\HH)$,  we say that $u$ is in the domain 
of the divergence operator $\de$    if there is a   square integrable random variable, denoted by $\de(u)$,   such  that 
\begin{equation}\label{dual}
\EE  \lc \de(u) F \rc =\EE  \lc \langle DF,u
\rangle_{\HH}\rc \quad \hbox{for any $F \in \mathbb{D}^{1,2}$}\,. 
\end{equation}
$\de(u)$ is also called the Skorohod integral of $u$  and we also denote this integral by $\de(u)=\int_{\RR_+\times \RR^d} 
u(t,x) W(dt, dx)$.

For any integer $n\ge 0$ we denote by $\mathbf{H}_n$ the $n$-th Wiener chaos of $W$. We recall that $\mathbf{H}_0$ is simply  $\R$ and for $n\ge 1$, $\mathbf {H}_n$ is the closed linear subspace of $L^2(\Omega)$ generated by the random variables $\{ H_n(W(\phi)),\phi \in \HH, \|\phi\|_{\HH}=1 \}$, where $H_n(x)
= e^{ \frac{x^2}{2}}\frac {d^n}{dx^n}e^{ -\frac{x^2}{2}}$ is the $n$-th Hermite polynomial.
For any $n\ge 1$, we denote by $\HH^{\otimes n}$ (resp. $\HH^{\odot n}$) the $n$-th tensor product (respectively the $n$-th  symmetric tensor product) of $\HH$. Then, the mapping $I_n(\phi^{\otimes n})= H_n(W(\phi))$ can be extended to a linear isometry between    $\HH^{\odot n}$ (equipped with the modified norm $\sqrt{n!}\| \cdot\|_{\HH^{\otimes n}}$) and $\mathbf{H}_n$. The chaos expansion theorem says that any square integrable  nonlinear functional $F$ of $W$ can be expressed a
\begin{equation}\label{eq:chaos-dcp}
F= \EE \lc F\rc + \sum_{n=1} ^\infty I_n(f_n),
\end{equation}
where the series converges in $L^2(\Omega)$, and the elements $f_n \in \HH ^{\odot n}$, $n\ge 1$, are determined by $F$.

The Skorohod integral (or divergence) of a random field $u$ can be
computed by  using the Wiener chaos expansion. More precisely,
suppose that $u=\{u(t,x) , (t,x) \in \R_+ \times\R\}$ is a random
field such that for each $(t,x)$, $u(t,x)$ is a  square-integrable random  variable, then  for each $(t,x)$ we have a Wiener chaos expansion of the form
\begin{equation}  \label{exp1}
u(t,x) = \EE \lc u(t,x) \rc + \sum_{n=1}^\infty I_n (f_n(\cdot,t,x)).
\end{equation}
Then  $u$ belongs to the domain of the divergence operator (that
is, $u$ is Skorohod integrable with respect to $W$) if and only if
the following series converges in $L^2(\Omega)$
\begin{equation}\label{eq:delta-u-chaos}
\delta(u)= \int_0 ^\infty \int_{\R^d}  u(t,x) \,   W(dt  , dx )
= W(\EE [u]) + \sum_{n=1}^\infty I_{n+1} (\widetilde{f}_n),
\end{equation}
where $\widetilde{f}_n$ denotes the symmetrization of $f_n$ in all its $n+1$ variables.

We shall use the following lemma  which has been used in  \cite{HHLNT},   \cite{HHLNT1}, \cite{HHNT}. 
\begin{lemma}\label{lem:intg-simplex}
For $m\ge 1$ let $\al   \in (-1 , 1)^m$  and  set $|\alpha |= \sum_{i=1}^m
\alpha_i  $. For $t\in\ott$, the $m$-th  dimensional simplex over $\ot$ is denoted by
$T_m(t)=\{(r_1,r_2,\dots,r_m) \in \R^m: 0<r_1  <\cdots < r_m < t\}$.
Then there is a constant $c>0$ such that
\[
J_m(t, \alpha):=\int_{T_m(t)}\prod_{i=1}^m (r_i-r_{i-1})^{\alpha_i}
dr \le \frac { c^m t^{|\alpha|+m } }{ \Gamma(|\alpha|+m +1)},
\]
where by convention, $r_0 =0$.
\end{lemma}

We shall also use the following classical result. The 
Mittag-Leffler function is defined as 
\[
E_\al(z)=\sum_{n=0}^\infty \frac{z^n}{\Gamma(\al n+1)}\,.
\]
It is well-known that  $|E_\al(z)|\le C e^{|z|^{\frac{1}{\al}}}$.
By Stirling formula one also have
\begin{equation}
 \sum_{n=1}^\infty \frac{\left|z\right|^n }{(n!)^\al} \le  C \exp\left\{|z|^{\frac{1}{\al}}\right\}\,. \label{e.mittag_leffler}
\end{equation}

We shall use $c$ and $C$ to denote constants. For example, we may write $c^{2n}$ as $c^n$.  

\section{Existence and uniqueness via chaos expansions}\label{subsec: chaos}
 
Let $u=\{u(t,x), 0 \leq t \leq T, x \in \mathbb{R}^d  \}$ be a real-valued   stochastic process  such that for all $t\in[0,T]$ and $x\in\R^d$ the process 
$\{p_{t-s}(x-y)u(s,y) \1_{[0,t]}(s), 0 \leq s \leq t, y \in \mathbb{R}^d\}$ is Skorohod integrable, where $p_t(x)=(2\pi t)^{-d/2} \exp(-\frac{|x|^2}{2})$   is the heat kernel on the real line related to $\frac{1}{2}\Delta$. 
\begin{definition}We say that $u$ is a mild solution of \eqref{spde} if for all $t \in [0,T]$ and $x\in \mathbb{R}^d$ we have
\begin{equation}\label{eq:mild-formulation sigma}
u(t,x)= p_t*u_0(x) + \int_0^t \int_{\mathbb{R}^d}p_{t-s}(x-y)u(s,y) W(ds,dy) \quad a.s.,
\end{equation}
where the stochastic integral is understood  in the    Skorohod  sense. 
\end{definition}

 We make the following assumptions on the covariance structure of the noise $\dot W$.
 \begin{hypothesis}\label{h.2.1} 
 There are positive constants $\al_0\in[0, 1]$ and $C$ such that
\begin{equation}
\ga_0(t)\le C |t|^{-\al_0}\quad \hbox{for all $t>0$}\,. 
\end{equation}  
We also allow $\al_0=1$ and in this case, we take $\ga_0(t)=\de(t)$
(the Dirac delta function).  
\end{hypothesis}

 
\begin{hypothesis}\label{h.2.2}  
Regarding the spectral density $\mu(\xi)$, we assume that it satisfies one of the following  
two   hypotheses. 
\begin{enumerate}
\item[(i)]\  There  are  positive constants $\al_i \in (-1, 0]$ ($i=1,\dots,d$) and  $C$ such that   
\begin{equation} 
  \mu(\xi) \le C\prod_{i=1}^d  |\xi_i|   ^{\al_i} \quad \hbox{for all $\xi\in \RR^d$}\,. \label{e.h.2.2i}
\end{equation} 
We denote  $\al=\al_1+\cdots+\al_d $ and we also assume $2\al_0+\al<4$.   

\item[(ii)] The space dimension $d=1$  and 
there are   positive constants $\al >0$,  $C$ such that   
\begin{equation} 
\al+\al_0<3/2\,,\quad {\rm and}\quad   \mu(\xi) \le C   |\xi |   ^{ \al } \quad \hbox{for all $\xi\in \RR $}\,. \label{e.h.2.2ib}
\end{equation} 
\end{enumerate}
\end{hypothesis}

Let us now state a new existence and uniqueness theorem for our equation of interest.

\begin{theorem}\label{thm:exist-uniq-chaos}
Suppose that the hypotheses \ref{h.2.1}  and \ref{h.2.2}    are satisfied  and   that the initial condition $u_0$ satisfies
\begin{equation}\label{cond:fu0}
\int_{\RR } \left[1+|\xi|^{\al/2}\right] e^{-s|\xi|^2} |\hat u_0(\xi)|d\xi \le C s^{-\beta}
\end{equation}
for some $\beta<1-\frac{\al_0}{2}$. 
Then there exists a unique    solution to equation \eqref{spde},
that is,  there is  a unique process $u$ 
such that  $p_{t-\cdot}(x-\cdot)u$ is Skorohod integrable for any
$(t,x)\in\ott\times\R^d$ and relation  \eqref{eq:mild-formulation sigma}
holds true.  Moreover, under the hypothesis \ref{h.2.2} (ii) 
we have the following moment   bounds for the solution:
\begin{equation}
\EE  |u(t,x) |^p \le   C\exp\left\{ c p^{\frac{3-\al}{1-\al}} t^{\frac{3-2\al_0 -\al}{1- \al} }    \right\} \,, \label{e.moment_bounds}
 \end{equation}
 where  $C$ and $c$ are two constants independent of $p$, $t$ and $x$.  
\end{theorem}

\begin{remark} 
(i) It is known \cite{HHNT} that the theorem is  true when the hypotheses \ref{h.2.1} and \ref{h.2.2} (i) hold.  When $\ga_0(t)$ is the Dirac delta function (namely, when the noise $\dot W$ is white in time) and
when hypothesis \ref{h.2.2} (i) holds, the theorem is also true  (\cite{HHLNT1}).  

(ii) Under hypotheses \ref{h.2.1} and \ref{h.2.2} (i), the moments bounds is given by (if the initial condition is bounded below by a positive constant) 
\begin{equation}   \label{es0}
C\exp\left(C t^{\frac{4-2\al_0-\al}{2-\al}}p^{\frac{4-\al}{2-\al}}\right)\leq
\EE \lc u_{t,x}^p\rc\leq  C^{\prime}\exp\left(C^{\prime}
t^{\frac{4-2\al_0-\al}{2-\al}}p^{\frac{4-\al}{2-\al}}\right)\,. 
\end{equation}
It appears that there are some discrepancy between  \eqref{e.moment_bounds} and \eqref{es0}. This kind of phenomena was also observed in \cite{CHKN}. 
\end{remark}

\begin{proof}   We only need to show the theorem under the hypotheses \ref{h.2.1} and \ref{h.2.2} (ii) with $\al_0\in (0, 1)$.
If a solution $u$ is square integrable, then according to the general chaos expansion   \eqref{eq:delta-u-chaos},     $u(t,x)$ admits the following Wiener chaos expansion
\begin{equation}\label{eq:chaos-expansion-u(tx)}
u(t,x)=\sum_{n=0}^{\infty}I_n(f_n(\cdot,t,x))\,,
\end{equation}
where for each $(t,x)$, $f_n(\cdot,t,x)$ is a symmetric element in
$\HH^{\otimes n}$.  
Hence,  if one iterates    \eqref{eq:mild-formulation sigma}, one can find an
explicit formula for the kernels $f_n$ for $n \geq 1$. 
\begin{multline}\label{eq:expression-fn}
f_n(s_1,x_1,\dots,s_n,x_n,t,x) 
=\frac{1}{n!}p_{t-s_{\si(n)}}(x-x_{\si(n)})\cdots p_{s_{\si(2)}-s_{\si(1)}}(x_{\si(2)}-x_{\si(1)})
p_{s_{\si(1)}}u_0(x_{\si(1)})\,,
\end{multline}
where $\si$ denotes the permutation of $\{1,2,\dots,n\}$ such that $0<s_{\si(1)}<\cdots<s_{\si(n)}<t$
(see also  for instance,  formula (4.4) in \cite{HN}, formula (3.8) in \cite{HHLNT1},  formula (3.3) in \cite{HHNT}).
Then, to show the existence and uniqueness of the solution it suffices to prove that for all $(t,x)$ we have
\begin{equation}\label{chaos}
\sum_{n=0}^{\infty}n!\|f_n(\cdot,t,x)\|^2_{\HH^{\otimes n}}< \infty\,.
\end{equation}

The    Fourier transform of $f_n(t,x)$ (with respect to the variables $x_1, \cdots, x_n$ in \eqref{eq:expression-fn})  is 
\begin{align*}
\hat  f_n(s_1,\xi_1,\dots,s_n,\xi_n,t,x)&=
\frac{c ^n}{n!}  \int_\RR g(t,x; s, \xi, \zeta) \hat u_0(\zeta) e^{-\frac {  s_{\sigma(1)}|\zeta|^2} 2} d\zeta,
\end{align*}
where we denote  $s_{\si(n+1)}=t$ by convention  and 
\begin{equation}
g_n(t,x; s, \xi, \zeta)=  e^{-\frac12 s_{\sigma(1)}|\zeta|^2}   \prod_{i=1}^n e^{-\frac{1}{2}(s_{\si(i+1)}-s_{\si(i)})|\xi_{\si(i)}+\cdots +
\xi_{\si(1)} -\zeta|^2}   { e^{-\iota x\cdot (\xi_{\sigma(n)}+ \cdots + \xi_{\sigma(1)}-\zeta)}}\,. \label{e.gn} 
\end{equation}
Hence,   we have
\begin{align*}
n!\| f_n(\cdot,t,x)\|_{\HH^{\otimes n}}^2 
&= \frac{c^{2n}
}{n!}\int_{[0,t]^{2n} }\int_{\RR^n}
 \int_{\RR^2} g(t,x; s, \xi, \zeta) \bar g(t,x; s', \xi, \zeta')  \\
 &\qquad \qquad  \ga_0(s-s') 
  dsds' \mu(\xi_1)d \xi_1\cdots\mu(\xi_n)d \xi_n 
  \hat u_0(\zeta)    d\zeta \hat u_0(\zeta')     d\zeta'\,,
\end{align*}
where $\ga_0(s-s')=\prod_{i=1}^n \ga_0(s_i-s_i')$.  
Using  Cauchy-Schwarz inequality we get
\begin{align}
n!\| f_n(\cdot,t,x)\|_{\HH^{\otimes n}}^2 
&\le
\frac{c^{2n}}{n!}
 \int_{\RR^2}   \int_{[0,t]^{2n} } \sqrt{\phi_n(s, \zeta)\phi_n(s', \zeta')} \ga_0(s-s')dsds'  
\left|\hat u_0(\zeta)\right| \left|\hat u_0(\zeta')\right| d\zeta d\zeta^{\prime} \nonumber\\
&=\frac{c^{2n}}{n!}
 \int_{\RR^2}   \int_{[0,t]^{2n} } \psi_n(s) \psi_n(s') \ga_0(s-s')dsds'  
 \,,\label{e.by_psi_n}
\end{align}
where
\begin{align*}
\phi_n(s, \zeta)
=\int_{\RR^{n }} |g_n(t,x; s, \xi, \zeta)|^2\mu(\xi_1)d\xi_1\cdots \mu(\xi_n)d \xi_n
\quad\textrm{and}\quad
\psi_n(s)=\int_{\RR} \sqrt{\phi_n(s, \zeta)} |\hat u_0(\zeta)| d\zeta\,. 
\end{align*}
From \eqref{e.gn} and the assumption \ref{h.2.2}, it follows that
\begin{align*}
\phi_n(s, \zeta)
\le \int_{\RR^{n }} e^{-s_{\si(1)}|\zeta|^2} 
\prod_{i=1}^n   e^{- (s_{\si(i+1)}-s_{\si(i)})|\xi_{\si(i)}+\cdots +
\xi_{\si(1)} -\zeta|^2}  |\xi_1|^\al\cdots|\xi_n|^\al d\xi_1\cdots d\xi_n\,.
\end{align*}
Making  the change of variables $\eta_{i}:=\xi_{\si(i)}+\cdots +
\xi_{\si(1)} -\zeta$,  and using  the trivial bound $|\eta_{i}-\eta_{i-1}|^{\al}\le |\eta_{i-1}|^{\al}+|\eta_{i}|^{\al}$, we obtain
\begin{align*}
\phi_n(s, \zeta)
&\le e^{-s_{\si(1)}|\zeta|^2}  \int_{\RR^n}  
\prod_{i=1}^n   e^{- (s_{\si(i+1)}-s_{\si(i)})|\eta_i|^2}   |\eta_1+\zeta| ^\al  \prod_{i=2}^n
  |\eta_i-\eta_{i-1}|^\al  d\eta_1\cdots d\eta_n\\
 &\le e^{-s_{\si(1)}|\zeta|^2}  \int_{\RR^n}  
\prod_{i=1}^n   e^{- (s_{\si(i+1)}-s_{\si(i)})|\eta_i|^2}   (|\eta_1|^\al+|\zeta|  ^\al)  \prod_{i=2}^n
  (|\eta_i|^\al +|\eta_{i-1}| ^\al)  d\eta_1\cdots d\eta_n\,. 
\end{align*}
Expanding  the product $\prod_{i=2}^{n} (|\eta_{i}|^{\al }+|\eta_{i-1}|^{\al })$ in the above integral  we obtain an expression of the form $\sum_{j\in D_{n}} \prod_{i=1}^{n} |\eta_{i}|^{j_{i}}$, where $D_{n}$ is a subset of multi-indices of $\left\{ (j_1, \cdots, j_n): j_i\in\{0, \al, 2\al\}\right\}$    of length $n-1$.  
The complete description of $D_{n}$ is not necessary. All we need are the following facts: $\text{Card}(D_{n})=2^{n-1}$ and for any $j\in  D_{n}$, $j_1\in\{0,\alpha\}$, $j_{i} \in \{0, \al, 2\al\}$ for $ i=2,\ldots, n$ and
\begin{equation*}
|j|\equiv \sum_{i=1}^{n} j_i = (n-1)\al\,.
\end{equation*}
This   expansion yields the following bound for $\phi_n(s, \zeta)$. 
\begin{eqnarray*}
\phi_n(s, \zeta) 
&\leq & |\zeta|^\al e^{-s_{\si(1)}|\zeta|^2}  \sum_{\alpha \in D_{n}}   \int_{\RR^n } \prod_{i=1}^n
e^{-  (s_{\si(i+1)}-s_{\si(i)})|\eta_{i}|^2} \prod_{i=1}^n |\eta_i|^{j_i }
 d\eta_1\cdots d\eta_n \\
&& \qquad +e^{-s_{\si(1)}|\zeta|^2} \sum_{\alpha \in D_{n}} \int_{\RR^n }\prod_{i=1}^n e^{-  (s_{\si(i+1)}-s_{\si(i)})|\eta_{i}|^2} |\eta_1|^{\al} \prod_{i=1}^n |\eta_i|^{j_i }
 d\eta_1\cdots d\eta_n  \,.
\end{eqnarray*}
Making  the  substitution  $\xi_{i}=  (s_{\sigma(i+1)}-s_{\sigma(i)} )^{1/2} \eta_{i}$ in the above integral, and noticing that $\int_{\R} e^{- \xi^{2}} |\xi|^{j_i}d\xi$ is bounded by a constant
for any $j_i>-1$, we see that
\begin{eqnarray}
\phi_{ n}(s, \zeta) 
&\le &   C  |\zeta|^{\al  }e^{-s_{\si(1)}|\zeta|^2 }
   \sum_{\alpha \in D
_{n}} \prod_{i=1}^n
 (s_{\sigma(i+1)}-s_{\sigma(i )}) ^{-\frac{1}{2}(1+j_i)}\nonumber \\
&&\qquad +C    e^{-s_{\si(1)}|\zeta|^2 }
   \sum_{\alpha \in D
_{n}}   (s_{\sigma(2)}-s_{\sigma( 1)} )^{-\frac{\al+j_1}{2}}\prod_{i=2}^n
 (s_{\sigma(i+1)}-s_{\sigma(i )} )^{-\frac{1}{2}(1+j_i)} \,. 
 \label{eqn:phi_n} 
\end{eqnarray}
Thus we have 
\begin{eqnarray*}
\psi_{ n}(s ) 
&\le &   C\int_\RR |\zeta|^{\al/2 }e^{-s_{\si(1)}|\zeta|^2/2}
|\hat u_0(\zeta)|  d\zeta    \sum_{\alpha \in D
_{n}} \prod_{i=1}^n
 (s_{\sigma(i+1)}-s_{\sigma(i )}) ^{-\frac{1}{4}(1+j_i)}\\
&&\qquad +C  \int_\RR e^{-s_{\sigma( 1)}|\zeta|^2/2}
|\hat u_0(\zeta)| d\zeta  \sum_{\alpha \in D
_{n}}   (s_{\sigma(2)}-s_{\sigma( 1)} )^{-\frac{\al+j_1}{4}}\prod_{i=2}^n
 (s_{\sigma(i+1)}-s_{\sigma(i)} )^{-\frac{1}{4}(1+j_i)} \,. 
\end{eqnarray*}
From the assumption on the initial condition \eqref{cond:fu0}  we see 
\begin{align}
\psi_{ n}(s ) 
&\le   C s_{\si(1)}^{-\be}    \sum_{\alpha \in D
_{n}} \prod_{i=1}^n
 (s_{\sigma(i+1)}-s_{\sigma(i )}) ^{-\frac{1}{4}(1+j_i)} 
 \nonumber \\&\quad+C  s_{\sigma( 1)}^{-\beta}  \sum_{\alpha \in D
_{n}}   (s_{\sigma(2)}-s_{\sigma(1)}  )^{-\frac{\al+j_1}{4}}\prod_{i=2}^n
 (s_{\sigma(i+1)}-s_{\sigma(i )})^{-\frac{1}{4}(1+j_i)} \,. \label{e.psi_bound}
\end{align}
Now we return to \eqref{e.by_psi_n}.   From  the Hardy-Littlewood inequality (see \cite[Inequality (2.4)]{HN} or \cite[Inequality (1.5)]{valkeila}), we see that
\begin{align*}
 n!\| f_n(\cdot,t,x)\|_{\HH^{\otimes n}}^2
 \le \frac{c^n}{n!}  \left(\int_{[0, t]^n} \psi_n(s)^{\frac{2}{2-\al_0}}ds \right)^{2-\al_0}\,. 
\end{align*}

Plugging \eqref{e.psi_bound} into the above expression and uing the inequality $(a_1+\cdots+a_m)^\al\le m^\al (a_1^\al+\cdots+a_m^\al)$ 
for all positive numbers $a_1, \cdots, a_m$ and positive $\al$, we end up with
\begin{eqnarray}
 n!\| f_n(\cdot,t,x)\|_{\HH^{\otimes n}}^2
 &\le& \frac{c^n}{n!}  \sum_{j\in D_n}\left(\int_{[0, t]^n}
 s_{\si(1)}^{-\frac{2\be}{2-\al_0}}     \prod_{i=1}^n
 (s_{\sigma(i+1)}-s_{\sigma(i )}) ^{-\frac{ (1+j_i)} {4-2\al_0}}    ds
 \right)^{2-\al_0}\nonumber \\
 &&\qquad + \frac{c^n}{n!}  \sum_{j\in D_n} \left(\int_{[0, t]^n}  
 s_{\sigma( 1)}^{-\frac{2\beta}{2-\al_0} }   
   (s_{\sigma(2)}-s_{\sigma(1)}  )^{-\frac{\al+j_1}{4-2\al_0}}\prod_{i=2}^n
 (s_{\sigma(i+1)}-s_{\sigma(i )})^{-\frac{  1+j_i}{4-2\al_0}}
  \right)^{2-\al_0}\nonumber\\
  &=& c^n (n!)^{1-\al_0}   \sum_{j\in D_n}\left(\int_{T_n(t)}
 s_{\si(1)}^{-\frac{2\be}{2-\al_0}}     \prod_{i=1}^n
 (s_{ i+1 }-s_{ i  }) ^{-\frac{ (1+j_i)} {4-2\al_0}}    ds
 \right)^{2-\al_0}\nonumber\\
 &&\qquad + c^n (n!)^{1-\al_0}  \sum_{j\in D_n} \left(\int_{T_n(t)}  
 s_{1}^{-\frac{2\beta}{2-\al_0} }   
   (s_{2}-s_{1}  )^{-\frac{\al+j_1}{4-2\al_0}}\prod_{i=2}^n
 (s_{ i+1 }-s_{ i  })^{-\frac{  1+j_i}{4-2\al_0}}\,. 
  \right)^{2-\al_0}\,. \nonumber\\ \label{e.fn_by_psi}
\end{eqnarray}
From our condition on $\beta$, we see that $\frac{2\be}{2-\al_0}<1$.
For $i\ge2$, the maximum value of $j_i$ is $2\al$  and in this case 
\[
\frac{  1+j_i}{4-2\al_0}=\frac{  1+2\al}{4-2\al_0}<1\,.
\]
Since $j_1\le \alpha$, 
\begin{equation*}
	\frac{\alpha+j_1}{4-2 \alpha_0}\le \frac{2 \alpha}{4-2 \alpha_0}<1\,.
\end{equation*}
It follows that all the integrals appeared in the above 
multiple integrals are finite.  Using Lemma \ref{lem:intg-simplex}, we see that 
\begin{eqnarray*}
 \int_{T_n(t)}
 s_{\si(1)}^{-\frac{2\be}{2-\al_0}}     \prod_{i=1}^n
 (s_{ i+1 }-s_{ i  }) ^{-\frac{ (1+j_i)} {4-2\al_0}}    ds
 &\le &\frac{c^nt^{\sum_{i=1}^n \left(1-\frac{1+j_i}{4-2\al_0}
 \right)+1-\frac{2\be}{2-\al_0}}}{ \Ga\left( \sum_{i=1}^n \left(1-\frac{1+j_i}{4-2\al_0}
 \right)+1-\frac{2\be}{2-\al_0}+1\right)}\\
 &\le& c^n t^{\tilde \al_0}t^ \frac{(3-2\al_0-\al)n}{4-2\al_0}  (n!)^{-\frac 
  { 3-2\al_0-\al }{4-2\al_0}}\,. 
\end{eqnarray*} 
In the same way we have 
\begin{eqnarray*}
 \int_{T_n(t)}
s_{1}^{-\frac{2\beta}{2-\al_0} }   
   (s_{2}-s_{1}  )^{-\frac{\al+j_1}{4-2\al_0}}\prod_{i=2}^n
 (s_{ i+1 }-s_{ i  })^{-\frac{  1+j_i}{4-2\al_0}}   ds 
 &\le& c^n t^{\tilde \al_1}t^ \frac{(3-2\al_0-\al)n}{4-2\al_0}  (n!)^{-\frac  {3-2\al_0-\al }{4-2\al_0}}\,. 
\end{eqnarray*} 
Here $\tilde \al_0$ and $\tilde \al_1$ are two constants, independent of $n$.  Inserting the above two bounds into \eqref{e.fn_by_psi}, we see that
\begin{eqnarray}
 n!\| f_n(\cdot,t,x)\|_{\HH^{\otimes n}}^2
 \le c^n t^{\tilde \al_2} (n!)^{1-\al_0-\frac{3-2\al_0-\al}{2}} t^{\frac{3-2\al_0 -\al}{2}n}\
 =c^n t^{\tilde \al_2} (n!)^{ -\frac{1-\al }{2}} t^{\frac{3-2\al_0 -\al}{2}n}\,.
 \label{e.3.18} 
\end{eqnarray}   
From hypercontractivity (\cite{hubook} and \cite{Nua}, we see that
\[
\|I_n(f_n(t,x))\|_p
\le \sqrt{p-1}^n   \|I_n(f_n(t,x))\|_2
\le c^n \sqrt{p }^n  t^{\tilde \al_2/2} (n!)^{ -\frac{ 1-\al }{4}} t^{\frac{3-2\al_0 -\al}{4}n}\,. 
\]
Applying the estimate \eqref{e.mittag_leffler}, we obtain
\[
\|u(t,x)\|_p\le \sum_{n=0}^\infty  p^{\frac{n}{2}}
 c^n t^{\tilde \al_2/2} (n!)^{ -\frac{1-\al }{4}} t^{\frac{3-2\al_0 -\al}{4}n}
 \le C\exp\left\{ c p^{\frac{2}{1-\al}} t^{\frac{3-2\al_0-\al}{1-\al} }    \right\} \,,
 \]
which implies \eqref{e.moment_bounds}. 
\end{proof}

\section{H\"older and joint H\"older continuity} 
In this section, we shall study the joint H\"older continuity of the solution $u(t,x)$ to \eqref{spde}.  The following lemma will be useful later.
  \begin{lemma}\label{lem:supzeta} Let $\mu$ satisfy the hypothesis \ref{h.2.2} (i). 
    For any $\beta\ge0$ and $s>0$, we have
    \begin{equation}\label{eqn:supzeta}
      \sup_{\zeta\in\RR^d}\int_{\RR^d}e^{- 2s|\xi-\zeta|^2}|\xi- \zeta|^{\beta}\mu(\xi)d \xi\lesssim s^{-\frac {\alpha+\beta}2}\,.
    \end{equation}
  \end{lemma}
  \begin{proof}
    Put $C_\beta:=\sup_{x>0} e^{-x}x^\beta$, which is a finite constant. We note that
    \begin{align*}
      \int_{\RR^d}e^{- 2s|\xi-\zeta|^2}|\xi- \zeta|^{\beta}\mu(\xi)d \xi
      &\le C_\beta s^{-\frac \beta2}\int_{\RR^d}e^{-s|\xi- \zeta|^2}\mu( \xi)d \xi
      \\&=C_\beta s^{-\frac \beta2}\int_{\RR^d}e^{-s|\xi- \zeta|^2}|\xi_1|^{\al_1}d\xi_1\cdots |\xi_d|^{\alpha_d}d \xi_d \,.
    \end{align*}
From Lemma A.1 of \cite{HNS}, we have
\begin{equation*}
	\int_{\RR^d}e^{-s|\xi- \zeta|^2}|\xi_1|^{\alpha_1}d\xi_1\cdots |\xi_d|^{\alpha_d}d \xi_d \lesssim s^{-\frac \alpha2}\,.
\end{equation*}
This implies the result.
  \end{proof}

To simplify notation  we denote $u_n(t,x)=I_{n,t} (f_n(t,x,\cdot))$,  where $f_n$ is the $n$-th chaos kernel defined in \eqref{eq:expression-fn} and  we 
explicitly express the dependence of the multiple integral of $f_n$ on  $t$:
\[
I_{n,t}(f_n(t,x))=\int_{[0, t]^n} f_n(t,x; s_s, \cdots, s_n, x_1, \cdots, x_n) 
W(ds_s, dx_1)\cdots W(ds_n, dx_n)
\,.
\] Thus, the solution
$u(t,x)$ can be written as its chaos expansion: $u(t,x)=\sum_{n=0}^\infty u_n(t,x)$.   
  \begin{proposition}\label{prop:unhder} Assuming that hypothesis \ref{h.2.1} holds   and $u_0$ satisfies \eqref{cond:fu0}. 
  \begin{enumerate}
  \item[(i)] Let  the hypothesis \ref{h.2.2} (i) be  satisfied  and 
  let $\bar  \alpha_0$ and $\bar\alpha$ be in $[0,1]$ such that
  \begin{equation}
  2\bar \alpha_0 +\bar \alpha <2- \alpha_0- \frac\alpha2\,. \label{e.al-be-de} 
  \end{equation}
  [It is easy to see the right hand side is positive under our assumption.] \ 
  Then,  there is a positive constant $c $ such that for every   $t\ge r\ge 0$ and $x,y\in\RR^d$
  \begin{equation}
  \EE \left|u_n(t,y)-u_n(r,y)-u_n(t,x)+u_n(r,x)\right|^2
  \le c^n  (n!)^{\frac{\alpha }{2}-1} 
  |t-r|^{2\bar \alpha_0} |x-y|^{2 \bar \alpha}  t^{\frac{4-2\al_0-\al}{2}n}    \,.
  \label{e.4.5}
  \end{equation}
  \item[(ii)] Let  the hypothesis \ref{h.2.2} (ii) be  satisfied  and 
  let $\bar  \alpha_0$ and $\bar\alpha$ be in $[0,1]$ such that
  \begin{equation}
  2\bar \alpha_0 +\bar \alpha <\frac{3-2\al_0-\al}{2}\,. \label{e.al-be-de} 
  \end{equation} 
  Then,  there is a positive constant $c $ such that for every   $t\ge r\ge 0$ and $x,y\in\RR^d$
  \begin{equation}
  \EE \left|u_n(t,y)-u_n(r,y)-u_n(t,x)+u_n(r,x)\right|^2
  \le c^n  (n!)^{\frac{\alpha -1}{2} } 
  |t-r|^{2\bar \alpha_0} |x-y|^{2 \bar \alpha}  t^{\frac{3-2\al_0-\al}{2}n}    \,.
  \label{e.4.5a}
  \end{equation}
  \end{enumerate} 
  \end{proposition}
  \begin{proof}  We follow the notation in the proof of 
  Theorem \ref{thm:exist-uniq-chaos} and divide the proof into several steps.
  
\noindent{\it Step 1}. \   
    For every $s=(s_1,\cdots,s_n)\in[0,t]^n$, the Fourier transform of $f_n(t,x,s_1,x_1,\dots,s_n,x_n)$ with respect to the spatial variables $(x_1,\dots,x_n)$, denoted by $\hat f_n(t,x,s,\xi)$ is 
    \begin{align*}
      \hat  f_n(t,x)&=\hat  f_n(t,x,s_1,\xi_1,\dots,s_n,\xi_n)
    \\&=
    \frac{c^n}{n!}   \int_{\RR ^d}
    \prod_{i=1}^n e^{-\frac{1}{2}(s_{\si(i+1)}-s_{\si(i)})|\xi_{\si(i)}+\cdots +
    \xi_{\si(1)} -\zeta|^2} 
     { e^{-\imath x \cdot(\xi_{\sigma(n)}+ \cdots + \xi_{\sigma(1)}-\zeta)}} \hat u_0(\zeta) e^{-\frac {  s_{\sigma(1)}|\zeta|^2} 2} d\zeta,
    \end{align*}
    where $\sigma$ is a permutation of $\{1,\dots,n\}$ such that $s_{\sigma(1)}\le\cdots\le s_{\sigma(n)}$ and we have set $s_{\si(n+1)}=t$. 
    We   use the following notation
    \[
    \eta_i=\xi_{\sigma(i)}+ \cdots + \xi_{\sigma(1)}
    \quad\textrm{and}\quad
    g_n(s, \eta, \zeta)
    =\prod_{i=1}^{n-1}  e^{-\frac{1}{2}(s_{\si(i+1)}-s_{\si(i)})|\eta_i -\zeta|^2}\,.
    \]
    Thus,  we can rewrite
    \begin{equation*}
    \hat  f_n(t,x)=
    \frac{c^n}{n!}   \int_{\RR ^d}   e^{-\frac{1}{2}(t-s_{\si(n)})|\eta_n -\zeta|^2 -\imath x\cdot (\eta_n-\zeta)}
     g_n(s,\eta, \zeta) \hat u_0(\zeta) e^{-\frac {  s_{\sigma(1)}|\zeta|^2} 2} d\zeta\,.
    \end{equation*}
    We denote $\Delta^n_{r,t}=[0,t]^n\setminus[0,r]^n$ and apply triangle inequality to get 
    \begin{align*}
      \EE|u_n(t,x)-u_n(t,y)-u_n(r,x)+u_n(r,y)|^2
      \le 2A_n+2B_n\,,
    \end{align*}
    where
    \begin{equation*}
      A_n=\EE \lt|I_{n,r} \lt( f_n(t,x,\cdot)-f_n(t,y,\cdot)  -f_n(r,x,\cdot)+f_n(r,y,\cdot) \rt)\rt|^2
    \end{equation*}
    and
    \begin{equation*}
      B_n=\EE \lt|I_{n,t} \lt([f_n(t,x,\cdot)-f_n(t,y,\cdot)]\1_{\Delta^n_{r,t}}\rt)\rt|^2\,.
    \end{equation*}
\noindent{\it Step 2}. \       For every $s,s'\in[0,r]^n$, we denote 
    \begin{align*}
      \phi(s,s')=\frac1{(2 \pi)^{nd}}&\int_{\RR^{nd}}\lt(  \hat f_n(t,x,s,\xi)-\hat f_n(t,y,s,\xi)  -\hat f_n(r,x,s,\xi)+\hat f_n(r,y,s,\xi)\rt)
      \\&\times\overline{\lt( \hat f_n(t,x,s',\xi)-\hat f_n(t,y,s',\xi) -\hat f_n(r,x,s',\xi)+\hat f_n(r,y,s',\xi)\rt)}\mu(\xi_1)d \xi_1\cdots \mu( \xi_n) d \xi_n
    \end{align*}
    and
     \begin{equation}
      \phi_1(s)=\frac1{(2 \pi)^{nd}}\int_{\RR^{nd}}\lt| \hat f_n(t,x,s,\xi)-\hat f_n(t,y,s,\xi)  -\hat f_n(r,x,s,\xi)+\hat f_n(r,y,s,\xi)\rt|^2\mu( \xi_1)d \xi_1\cdots \mu( \xi_n)d \xi_n\,.\label{e.phi_1} 
    \end{equation}
    Applying It\^o isometry of Wiener chaos and Cauchy-Schwarz, we have
    \begin{align*}
      A_n=n!\int_{[0,r]^n}\int_{[0,r]^n} \phi(s,s')\prod_{j=1}^n \gamma_0(s_j-s_j')dsds'\le n!\int_{[0,r]^n}\int_{[0,r]^n}[\phi_1(s)\phi_1(s')]^{\frac12}\prod_{j=1}^n \gamma_0(s_j-s_j')dsds'\,.
    \end{align*}
    Using  the Hardy-Littlewood inequality (see \cite[Inequality (2.4)]{HN} or \cite[Inequality (1.5)]{valkeila}) we see that
    \begin{align}\label{eqn:CnHardy}
      A_n\lesssim c^nn!\lt(\int_{[0,r]^n}|\phi_1(s)|^{\frac1{2- \alpha_0}} ds\rt)^{2- \alpha_0}.
    \end{align}
\noindent{\it Step 3}. \   We estimate bound $A_n$ under hypothesis \ref{h.2.2} (i).  Observe that for every $s\in[0,r]^n$, 
    \begin{align*}
      & \hat f_n(t,x,s ,\xi)-\hat f_n(t,y,s ,\xi) -\hat f_n(r,x,s ,\xi)+\hat f_n(r,y,s ,\xi)
      \\&=\frac{c^n}{n!}\int_{\RR^d}e^{-\frac12(r-s_{\sigma(n)})|\eta_n- \zeta|^2-\imath y \cdot( \eta_n-\zeta)}
     \lt[e^{-\frac12(t-r)|\eta_n- \zeta|^2}-1 \rt]\lt[e^{-\imath(x-y)\cdot(\eta_n- \zeta)}-1 \rt]g_n(s,\eta,\zeta)  \hat u_0(\zeta) e^{-\frac {  s_{\sigma(1)}|\zeta|^2} 2} d\zeta.
    \end{align*}
    Applying Jensen's inequality with respect to the measure $|\hat u_0(\zeta)|e^{-\frac{s_{\sigma(1)}|\zeta|^2}2}d \zeta$, we obtain
    \begin{align*}
      &\lt| \hat f_n(t,x,s ,\xi)-\hat f_n(t,y,s ,\xi) -\hat f_n(r,x,s ,\xi)+\hat f_n(r,y,s ,\xi)\rt|^2
      \\ &\le\frac{c^{ n}}{(n!)^2}\int_{\RR^d}|\hat u_0(\zeta)| e^{-\frac {  s_{\sigma(1)}|\zeta|^2} 2} d\zeta
        \times\int_{\RR^d}e^{-(r-s_{\sigma(n)})|\eta_n- \zeta|^2}
     \\ &\quad \lt|e^{-\frac12(t-r)|\eta_n- \zeta|^2}-1 \rt|^2\lt|e^{-\imath(y-x)\cdot(\eta_n- \zeta)}-1 \rt|^2 |g_n(s,\eta,\zeta)|^2 |\hat u_0(\zeta)| e^{-\frac {  s_{\sigma(1)}|\zeta|^2} 2} d\zeta.
    \end{align*}
    Applying the elementary estimates 
    \begin{equation*}
        |e^{-\frac12(t-r)|\eta_n- \zeta|^2}-1 |\lesssim |t-r|^{\bar \alpha_0}|\eta_n- \zeta|^{2\bar \alpha_0}\,, 
        \quad \quad
        |e^{-\imath(y-x)\cdot(\eta_n- \zeta)}-1|\lesssim |y-x|^{\bar \alpha}|\eta_n- \zeta|^{\bar \alpha}
    \end{equation*} 
 for any $\tilde \al_0, \tilde \al\in [0, 1]$   and noticing \eqref{cond:fu0}, we see that
    \begin{align}
      \phi_1(s)&\lesssim \frac{c^n}{(n!)^2}|t-r|^{2 \bar\alpha_0}|x-y|^{2\bar \alpha} s_{\sigma(1)}^{-\beta} 
      \int_{\RR^d}\int_{\RR^{nd}} e^{-(r-s_{\sigma(n)}) |\eta_n- \zeta|^2}
      \nonumber
      \\ &\quad|\eta_n- \zeta|^{4 \bar \alpha_0+2\bar \alpha} |g_n(s,\eta,\zeta)|^2  \mu( \xi_1)d \xi_1\cdots \mu(\xi_n)d \xi_n |\hat u_0(\zeta)| e^{-\frac {  s_{\sigma(1)}|\zeta|^2} 2} d\zeta\,.\label{eqn:phi1}
    \end{align}
    To estimate the integral above, we first consider the integration with respect to $\xi_{\sigma(n)}$. Noting that $g_n$ does not depend on $\xi_{\sigma(n)}$, we can apply Lemma \ref{lem:supzeta} to get
    \begin{align*}
      \int_{\RR^d}e^{-(r- s_{\sigma(n)})|\eta_n- \zeta|^2}|\eta_n- \zeta|^{4\bar \alpha_0+2\bar \alpha}\mu(\xi_{\sigma(n)})d\xi_{\sigma(n)}
      \lesssim  (r- s_{\sigma(n)})^{-2 \bar \alpha_0-\bar \alpha-\frac{\alpha}2}\,,
    \end{align*}
    where the implied constant is independent of $\zeta$ and $\xi_1,\dots,\xi_{\sigma(n-1)}$. Next, we integrate the variable $\xi_{\sigma(n-1)}$.
    Due to the factor $e^{-(s_{\sigma(n)- s_{\sigma(n-1)}})|\eta_{n-1}- \zeta |^2}$ in $g_n$, this amounts to estimate an analogous integral but with $\bar \alpha_0=\bar \alpha=0$. We continue this way until all the variables $\xi_{\sigma(1)},\dots, \xi_{\sigma(n)}$ are integrated out. The last integral is $\int_{\RR^d}e^{-\frac12 s_{\sigma(1)}|\zeta|^2}|\hat u_0(\zeta)|d \zeta$ which is controlled by \eqref{cond:fu0}. Therefore, we have
    \begin{align*}
      \phi_1(s)\lesssim \frac{c^n}{(n!)^2}|t-r|^{2\bar \alpha_0}|x-y|^{2\bar \alpha}(r-s_{\sigma(n)})^{-2\bar \alpha_0-\bar \alpha-\frac \alpha2}\prod_{j=1}^{n-1}(s_{\sigma(j+1)}- s_{\sigma(j)})^{-\frac \alpha2} s_{\sigma(1)}^{-2 \be}\,.
    \end{align*}
    Together with \eqref{eqn:CnHardy}, we obtain
    \begin{eqnarray} 
      A_n
      &\lesssim& \frac{c^n}{n!}|t-r|^{2\bar \alpha_0}|x-y|^{2\bar \alpha}\lt[\int_{[0,r]^n}(r-s_{\sigma(n)})^{-\frac{2\bar \alpha_0+\bar \alpha+\frac \alpha2}{2- \alpha_0}}\prod_{j=1}^{n-1}(s_{\sigma(j+1)}- \sigma_{\sigma(j)})^{-\frac \alpha{2(2- \alpha_0)} }s_{\sigma(1)}^{- \frac{2\beta}{2- \alpha_0}} ds\rt]^{2- \alpha_0}
      \nonumber\\&\lesssim& {c^n}{(n!)^{1-\al_0}}|t-r|^{2\bar \alpha_0}|x-y|^{2\bar \alpha}\lt[\int_{T_n(r)}(r-s_{n})^{-\frac{2\bar \alpha_0+\bar \alpha+\frac \alpha2}{2- \alpha_0}}\prod_{j=1}^{n-1}(s_{j+1}- s_{j})^{-\frac \alpha{2(2- \alpha_0)} }s_{1}^{- \frac{2\beta}{2- \alpha_0}} ds\rt]^{2- \alpha_0}\,, \label{e.4.7add} 
    \end{eqnarray}  
where $T_n(r)=\{(s_1,\dots,s_n) :s_1<\cdots<s_n<r \}$. Using the condition \eqref{e.al-be-de} and Lemma \ref{lem:intg-simplex}, we see that the integration over $T_n(r)$ above is finite and is bounded by
\[
r^{n\frac{2- \alpha_0-\frac \alpha2}{2- \alpha_0}-\frac{2\bar \alpha_0+\bar \alpha+2 \beta}{2- \alpha_0}}\left[\Gamma\left(n\frac{2- \alpha_0- \frac\alpha2}{2- \alpha_0}+1\right)\right]^{-1}\,.
\]
 Using Stirling formula, we have 
    \begin{eqnarray}  
      A_n
      &\lesssim & c^n|t-r|^{2\bar \alpha_0}|x-y|^{2 \bar \alpha}r^{n(2- \alpha_0- \frac\alpha2)-2\bar \alpha_0- \bar \alpha- 2\beta}\frac{(n!)^{1- \alpha_0}}{\Gamma(n\frac{2- \alpha_0- \frac\alpha2}{2- \alpha_0}+1)^{2- \alpha_0} }
  \nonumber    \\&\lesssim& c^n|t-r|^{2\bar \alpha_0}|x-y|^{2 \bar \alpha}r^{n(2- \alpha_0- \frac\alpha2)-2\bar \alpha_0- \bar \alpha- 2\beta}\frac{(n!)^{1- \alpha_0}}{\Gamma(n(2- \alpha_0- \frac\alpha2)+1)}
   \nonumber   \\&\lesssim& c^n|t-r|^{2\bar \alpha_0}|x-y|^{2 \bar \alpha}r^{n(2- \alpha_0- \frac\alpha2)-2\bar \alpha_0- \bar \alpha- 2\beta}(n!)^{\frac \alpha2-1}\,.\label{e.4.8add}
    \end{eqnarray} 
\noindent{\it Step 4}. \  Now we consider $A_n$ for 
$d=1$ and $\al>0$.    In this case the inequality \eqref{eqn:phi1}
 becomes
 \begin{align}
      \phi_1(s)&\lesssim \frac{c^n}{(n!)^2}|t-r|^{2 \bar\alpha_0}|x-y|^{2\bar \alpha} s_{\sigma(1)}^{-\beta} 
     \int_{\RR }\int_{\RR^{n }} e^{-(r-s_{\sigma(n)})|\eta_n- \zeta|^2}|\eta_n- \zeta|^{4 \bar \alpha_0+2\bar \alpha} \nonumber
      \\ &\quad\prod_{i=1}^{n-1}  e^{- (s_{\si(i+1)}-s_{\si(i)})|\eta_i -\zeta|^2}   \prod_{i=1}^n |\xi_i|^{\al} d\xi_1\cdots d\xi_n |\hat u_0(\zeta)| e^{-\frac {  s_{\sigma(1)}|\zeta|^2} 2} d\zeta\,. 
    \end{align} 
Making substitution $\eta_i-\zeta=\xi_{\si(i)}+\cdots+\xi_{s(\si(1)}-\zeta=z_i$, we have 
\begin{align}
      \phi_1(s)&\lesssim \frac{c^n}{(n!)^2}|t-r|^{2 \bar\alpha_0}|x-y|^{2\bar \alpha} s_{\sigma(1)}^{-\beta} 
     \int_{\RR }\int_{\RR^{n }} e^{-(r-s_{\sigma(n)})|z_n |^2} \prod_{i=1}^{n-1}  e^{- (s_{\si(i+1)}-s_{\si(i)})|z_i  |^2}   \nonumber
      \\ &\quad |z_n |^{4 \bar \alpha_0+2\bar \alpha}  |z_1+\zeta|^\al \prod_{i=2}^n |z_i-z_{i-1}|^{\al} dz_1\cdots dz_n |\hat u_0(\zeta)| e^{-\frac {  s_{\sigma(1)}|\zeta|^2} 2} d\zeta\nonumber\\
&\lesssim \frac{c^n}{(n!)^2}|t-r|^{2 \bar\alpha_0}|x-y|^{2\bar \alpha} s_{\sigma(1)}^{-\beta} 
     \int_{\RR }\int_{\RR^{n }} e^{-(r-s_{\sigma(n)})|z_n |^2} \prod_{i=1}^{n-1}  e^{- (s_{\si(i+1)}-s_{\si(i)})|z_i  |^2}   \nonumber
      \\ &\quad |z_n |^{4 \bar \alpha_0+2\bar \alpha}  (|z_1|^\al+|\zeta|^\al) \prod_{i=2}^{n } (|z_i|^\al +|z_{i-1}|^{\al}) dz_1\cdots dz_n |\hat u_0(\zeta)| e^{-\frac {  s_{\sigma(1)}|\zeta|^2} 2} d\zeta      \,. 
    \end{align} 
As for \eqref{eqn:phi_n}, we expand     $(|z_1|^\al+|\zeta|^\al) \prod_{i=2}^n (|z_i|^\al +|z_{i-1}|^{\al})$ and integrate $dz_1\cdots dz_n   d\zeta$ to obtain 
\begin{equation} 
\phi_{1}(s ) 
\le  \frac{c^n}{(n!)^2}|t-r|^{2 \bar\alpha_0}|x-y|^{2\bar \alpha} s_{\sigma(1)}^{-2\beta}   (r-s_{\si(n)})^{-2\tilde \al_0-\tilde \al}
   \sum_{j  \in D
_{n}} \prod_{i=1}^n
 (s_{\sigma(i+1)}-s_{\sigma(i )}) ^{-\frac{1}{2}(1+j_i)}\,. 
  \label{e.phi_bound}
\end{equation}
The set $D_n$ is similar to the one defined in  Section 3 with a minor difference. More specifically, $D_{n}$ is a set of multi-indices with the following properties: $\text{Card}(D_{n})=2^{n-1}$, for each $j=(j_1,\dots,j_n)\in D_n$, $j_n\in\{0,\alpha\}$, $j_i\in\{0,\alpha,2 \alpha\}$ and  
\begin{equation*}
|j|\equiv \sum_{i=1}^{n} j_i = (n-1)\al.
\end{equation*}
Similar to the estimates \eqref{e.fn_by_psi}-\eqref{e.3.18}  if we assume 
$2\tilde \al_0+\tilde \al<\frac{3-2\al_0-\al}{2}$,  then we have
\begin{eqnarray}
A_n
&\le& c^n (n!)^{1-\al_0} |t-r|^{2\tilde \al_0} |x-y|^{2\tilde \al} \nonumber\\
&&\qquad 
\sum_{j\in D_n} \left[\int_{[0, r]^{nd}} \left( s_{\sigma(1)}^{-2\beta}   (r-s_{\si(n)})^{-2\tilde \al_0-\tilde \al-\frac{1+\al}{2}}
   \sum_{j  \in D
_{n}} \prod_{i=1}^{n-1}
 (s_{\sigma(i+1)}-s_{\sigma(i )}) ^{-\frac{1}{2}(1+j_i)}\right)^{\frac1{2-\al_0}}\right]^{2-\al_0}\nonumber\\
&\le&  c^n (n!)^{\frac{\al-1}{2} } |t-r|^{2\tilde \al_0} |x-y|^{2\tilde \al}
t  ^{\tilde \al_3} t^{\frac{3-2\al_0-\al}{2} n}\,.  \label{e.4.12} 
\end{eqnarray}

\noindent{\it Step 5}. We estimate $B_n$ under hypothesis \ref{h.2.2} (i).  
    As in deriving \eqref{eqn:CnHardy}, we apply It\^o isometry of Wiener chaos, Cauchy-Schwarz inequality and Hardy-Littlewood inequality to obtain a similar inequality to  \eqref{eqn:CnHardy}. Namely, 
    \begin{equation*}
      B_n\le c^n n!\lt(\int_{\Delta^n_{r,t}}|\psi_1(s)|^{\frac1{2- \alpha_0}}ds\rt)^{2- \alpha_0}\,,
    \end{equation*}
    where
    \begin{equation*}
      \psi_1(s)=\frac1{(2 \pi)^{nd}}\int_{\RR^{nd}} \lt|\hat f_n(t,x,s,\xi)-\hat f_n(t,y,s,\xi) \rt|^2 \mu(\xi_1)d \xi_1\cdots \mu(\xi_n)d \xi_n\,.
    \end{equation*}
    Similar to \eqref{e.phi_bound}, we have
\begin{align}
&\psi_1(s)
\nonumber\\&\le \frac{c^n}{(n!)^2} |x-y|^{2\tilde \al} s_{\si(1)}^{-\be} \int_{\RR^{(n+1)d}}
|\eta_n-\zeta|^{2\tilde \al} e^{-(t-s_{\si(n)})|\eta_n-\zeta|^2} |g_n(s,\eta, \zeta) | \mu(\xi_1)d\xi_1\cdots\mu(\xi_n)d \xi_n |\hat u_0(\zeta)| e^{-\frac{s_{\si(1)}|\zeta|^2}{2}} d\zeta\nonumber\\
&\le \frac{c^n}{(n!)^2}|x-y|^{2\tilde \al}  s_{\si(1)}^{-2\be} (t-s_{\si(n)})^{-\frac{\al}{2}-\tilde \al}\prod_{i=1}^{n-1} (s_{\si(i+1)}- s_{\si(i+1)})^{ -\frac{\al}{2}}\,. 
\end{align}
Thus we have 
 \begin{equation*}
      B_n\le \frac{c^n}{n!} |x-y|^{2\tilde \al} \lt(\int_{\Delta^n_{r,t}}
    \left( s_{\si(1)}^{-2\be} (t-s_{\si(n)})^{-\frac{\al}{2}-\tilde \al}\prod_{i=1}^{n-1} (s_{\si(i+1)}- s_{\si(i+1)})^{ -\frac{\al}{2}}    
 \right)^{\frac1{2-\al_0}}     ds\rt)^{2- \alpha_0}\,,
    \end{equation*}
The integrating region $\Delta^n_{r,t}$ can be  decomposed as
$\Delta^n_{r,t}=\cup_{\si}\cup_{0\le k\le n}
T_{n,\si, k, r, t}$,
where $\si$ is a permutation of $\{1 ,2, \cdots, n\}$ and 
\[
T_{n,\si, k, r, t}=\left\{0<s_{\si(1)}<\cdots <s_{\si(k)}<r<
s_{\si(k+1)}<\cdots<s_{\si(n)}<t\right\}\,. 
\]
When $\sigma$ is the identity, we denote $T_{n,  k, r, t}=T_{n,\si, k, r, t}$.
Thus by symmetry 
 \begin{equation*}
      B_n\le (n!)^{1-\al_0} c^n |x-y|^{2\tilde \al} \lt(\sum_{k=0}^n\int_{T_{n,  k, r, t}}
    \left( s_1^{-2\be} (t-s_n )^{-\frac{\al}{2}-\tilde \al}\prod_{i=1}^{n-1} (s_{ i+1 }- s_{ i })^{ -\frac{\al}{2}}    
 \right)^{\frac1{2-\al_0}}     ds\rt)^{2- \alpha_0}\,,
    \end{equation*}
For each $k$, the multiple integral over $T_{n,k,r,t}$
is estimated as follows
\begin{eqnarray*}
&&\int_{T_{n,  k, r, t}}
    s_1^{-\frac{2\be}{2-\al_0} } (t-s_n )^{-\frac{ \al+2\tilde \al}{2-\al_0} }\prod_{i=1}^{n-1} (s_{ i+1 }- s_{ i })^{ -\frac{ \al}{2-\al_0} }   
      ds_1\cdots ds_n\\
&&\qquad\le \int_{0<s_1<\cdots<s_k<r }  
s_1^{-\frac{2\be}{2-\al_0} } (r-s_k)^{ -\frac{ \al}{2-\al_0} } \prod_{i=1}^{k-1} (s_{ i+1 }- s_{ i })^{ -\frac{ \al}{2-\al_0} } ds_1\cdots ds_k     \\
&&\qquad  \int_{ r<s_{k+1}<\cdots<s_n<t}   \prod_{i=k+1}^{n-1} (s_{ i+1 }- s_{ i })^{ -\frac{ \al}{2-\al_0} }          (t-s_n )^{-\frac{ \al+2\tilde \al}{2-\al_0} }s_{k+1}\cdots ds_n \\
&&\qquad\le \frac{c^n r^{k(1-\frac{\al}{2(2-\al_0)}+1-\frac{2\be}{2(2-\al_0)}}}
{\Ga \left(k(1-\frac{\al}{2(2-\al_0)})+1-\frac{2\be}{2(2-\al_0)} +1\right)}
\frac{  (t-r)^{(n-k-1)(1-\frac{\al}{2(2-\al_0)}+ 1-\frac{\al+2\tilde \al}{2(2-\al_0)} }}
{\Ga \left((n-k-1) (1-\frac{\al}{2(2-\al_0)})+1- \frac{\al+2\tilde \al}{2(2-\al_0)}+1\right)  } \\
&&\qquad\le (n!)^{\frac{\al}{2(2-\al_0)}-1} t^{(n-1)(1-\frac{\al}{2(2-\al_0)})   +1-\frac{2\be}{a(2-\al_0)}} (t-r)^{ 1-\frac{\al+2\tilde \al}{2(2-\al_0)}}\,. 
\end{eqnarray*}
Thus we have the following bound 
    \begin{align*}
      B_n
 &\lesssim c^n(n!)^{\frac \alpha2-1}|x-y|^{2\bar \alpha}|t-r|^{\frac{4-2\al_0-\al-2\tilde \al}{2} }t^{\tilde \be} t^{\frac{4 - 2\alpha_0-  \alpha}{2}n}   \,.
    \end{align*}
Denote 
\[
\al_0'=\frac{4-2\al_0-\al-2\tilde \al}{4}\,.
\]
It is easy to verify that $2\al_0'+\tilde \al=2-\al_0-\frac{\al}{2}$  and
\begin{equation}
      B_n
 \lesssim c^n(n!)^{\frac \alpha2-1}|x-y|^{2\bar \alpha}|t-r|^{2\tilde \al_0  }t^{\tilde \be} t^{\frac{4 - 2\alpha_0-  \alpha}{2}n}   \,.
    \end{equation}

\noindent{\it Step 6}. \
  We estimate  $B_n$ under hypothesis \ref{h.2.2} (i).  
  First we have 
    \begin{equation*}
      B_n\le c^nn!\lt(\int_{\Delta^n_{r,t}}|\psi_1(s)|^{\frac1{2- \alpha_0}}ds\rt)^{2- \alpha_0}\,,
    \end{equation*}
    where  
\begin{eqnarray*}
\psi_1(s)
&\le& \frac{c^n}{(n!)^2} |x-y|^{2\tilde \al} s_{\si(1)}^{-\be} \int_{\RR^{(n+1}d}
|\eta_n-\zeta|^{2\tilde \al} e^{-(t-s_{\si(n)})|\eta_n-\zeta|^2} |g_n(s,\eta, \zeta)| \mu(d\xi_1)\cdots\mu(d\xi_n) |\hat u_0(\zeta)| e^{-\frac{s_{\si(1)}|\zeta|^2}{2}} d\zeta\\
&\le& \frac{c^n}{(n!)^2}|x-y|^{2\tilde \al}  s_{\si(1)}^{-2\be} 
\int_{\RR^{(n+1}d}
|\eta_n-\zeta|^{2\tilde \al} e^{-(t-s_{\si(n)})|\eta_n-\zeta|^2}
\prod_{i=1}^{n-1} e^{-(s_{\si(i+1)}-s_{\si(i )})|\eta_i-\zeta|^2}  
\\
& & \qquad \qquad \qquad \prod_{i=1}^n |\xi_i|^\al d\xi_1 \cdots d\xi_n  |\hat u_0(\zeta) e^{-\frac{s_{\si(1)}|\zeta|^2}{2}} d\zeta \\
&\le& \frac{c^n}{(n!)^2} |x-y|^{2\tilde \al}  s_{\si(1)}^{-2\be} 
\int_{\RR^{(n+1}d}
|\eta_n-\zeta|^{2\tilde \al} e^{-(t-s_{\si(n)})|\eta_n-\zeta|^2}
\prod_{i=1}^{n-1} e^{-(s_{\si(i+1)}-s_{\si(i )})|\eta_i-\zeta|^2}  
\\
& & \qquad \qquad \qquad |\eta_1+\zeta|^\al \prod_{2=1}^n |\eta_i-\eta_{i-1} |^\al
| d\eta_1 \cdots d\eta_n  |\hat u_0(\zeta) e^{-\frac{s_{\si(1)}|\zeta|^2}{2}} d\zeta  \\
&\le& \frac{c^n}{(n!)^2}|x-y|^{2\tilde \al}  s_{\si(1)}^{-2\be} 
\int_{\RR^{(n+1}d} (
|\eta_n|^{2\tilde \al}+| \zeta|^{2\tilde \al}) e^{-(t-s_{\si(n)})|\eta_n-\zeta|^2}
\prod_{i=1}^{n-1} e^{-(s_{\si(i+1)}-s_{\si(i )})|\eta_i-\zeta|^2}  
\\
& & \qquad \qquad \qquad (|\eta_1|^{\al}+|\zeta|^\al) \prod_{2=1}^n (|\eta_i|^\al+| \eta_{i-1} |^\al)
| d\eta_1 \cdots d\eta_n  |\hat u_0(\zeta) e^{-\frac{s_{\si(1)}|\zeta|^2}{2}} d\zeta     \,. 
\end{eqnarray*}
In the same  way as for \eqref{e.phi_bound} and \eqref{e.4.12}, we have 
    \begin{align*}
      B_n
 &\lesssim c^n (n!)^{\frac {\alpha -1}2}|x-y|^{2\bar \alpha}|t-r|^{\frac{3-2\al_0-2\tilde \al-\al}{2}}
 t^{\tilde \be} t^{ \frac{3-2\al_0-\al}{2}n}  \,.
    \end{align*}
Denote $\tilde \al_0=\frac{3-2\al_0-2\tilde \al-\al}{4}$,  then we see easily that $2\tilde \al_0+\tilde \al<\frac{3-2\al_0-\al}{2}$ and
\begin{equation} 
      B_n
 \lesssim c^n (n!)^{\frac {\alpha -1}2}|x-y|^{2\bar \alpha}|t-r|^{2\tilde \al_0}
 t^{\tilde \be} t^{ \frac{3-2\al_0-\al}{2} n}  \,.
    \end{equation}
    This proves the proposition. 
  \end{proof}
   \begin{theorem}\label{thm:ujointhder} Assuming that  hypothesis \ref{h.2.1} holds   and $u_0$ satisfies \eqref{cond:fu0}. 
  \begin{enumerate}
  \item[(i)] Let  the hypothesis \ref{h.2.2} (i) be  satisfied  and 
  let $\bar  \alpha_0$ and $\bar\alpha$ be in $[0,1]$ such that
  \begin{equation*}
  2\bar \alpha_0 +\bar \alpha <2- \alpha_0- \frac\alpha2\,.  
  \end{equation*}
  [It is easy to see the right hand side is positive under our assumption.] \ 
  Then  there is a positive constant $c $ such that for every   $t\ge r\ge 0$ and $x,y\in\RR^d$
  \begin{equation}
  \EE \left|u(t,y)-u(r,y)-u(t,x)+u(r,x)\right|^p
  \le  C\exp\left(c  p^{\frac{4- \alpha}{2-\alpha} }t^{\frac{4-2\al_0- \alpha}{2-\alpha}}  \right)
  |t-r|^{ p\bar\al_0}  |x-y|^{ p\bar \alpha}   
       \,.
  \label{e.4.5aa}
  \end{equation}
  \item[(ii)] Let  the hypothesis \ref{h.2.2} (ii) be  satisfied  and 
  let $\bar  \alpha_0$ and $\bar\alpha$ be in $[0,1]$ such that
  \begin{equation*}
  2\bar \alpha_0 +\bar \alpha <\frac{3-2\al_0-\al}{2}\,.  
  \end{equation*} 
  Then  there is a positive constant $c $ such that for every   $t\ge r\ge 0$ and $x,y\in\RR^d$
    \begin{equation}
  \EE \left|u(t,y)-u(r,y)-u(t,x)+u(r,x)\right|^p
  \le  C\exp\left(c  p^{\frac{3- \alpha}{1-\alpha} }t^{\frac{3-2\al_0- \alpha}{1-\alpha} }  \right)
  |t-r|^{ p\bar\al_0}  |x-y|^{ p\bar \alpha}   
       \,.
  \label{e.4.5aaa}
  \end{equation} 
  \end{enumerate} 
  \end{theorem}
  
  The above theorem is about the joint H\"older continuity.
  For the usual H\"older continuity the argument is similar and simpler.
  We state it as follows. 
   \begin{theorem}\label{thm:uhder} Assuming that the hypothesis \ref{h.2.1}   and $u_0$ satisfies \eqref{cond:fu0}. 
  \begin{enumerate}
  \item[(i)] Let  the hypothesis \ref{h.2.2} (i) be  satisfied  and 
  let $\bar  \alpha_0$ and $\bar\alpha$ be in $[0,1]$ such that
  \begin{equation*}
  \max( 2\bar \alpha_0, \bar \alpha) <2- \alpha_0- \frac\alpha2\,.  
  \end{equation*}
  Then  there is a positive constant $c $ such that for every   $t\ge r\ge 0$ and $x,y\in\RR^d$
  \begin{equation}
  \EE \left|u(t,x)-u(r,x)|+|u(t,x)-u(t,y)|\right|^p
  \le  C\exp\left(c  p^{\frac{4- \alpha}{2-\alpha} }t^{\frac{4-2\al_0- \alpha}{2-\alpha}}  \right)
 \left[ |t-r|^{ p\bar\al_0} + |x-y|^{ p\bar \alpha}   \right]
       \,.
  \label{e.4.5aa}
  \end{equation}
  \item[(ii)] Let  the hypothesis \ref{h.2.2} (ii) be  satisfied  and 
  let $\bar  \alpha_0$ and $\bar\alpha$ be in $[0,1]$ such that
  \begin{equation*}
  \max(2\bar \alpha_0, \bar \alpha) <\frac{3-2\al_0-\al}{2}\,.  
  \end{equation*} 
  Then  there is a positive constant $c $ such that for every   $t\ge r\ge 0$ and $x,y\in\RR^d$
    \begin{equation}
\EE \left|u(t,x)-u(r,x)|+|u(t,x)-u(t,y)|\right|^p
  \le  C\exp\left(c  p^{\frac{3- \alpha}{1-\alpha} }t^{\frac{3-2\al_0- \alpha}{1-\alpha} }  \right)
  |t-r|^{ p\bar\al_0}  |x-y|^{ p\bar \alpha}   
       \,.
  \label{e.4.5aaa}
  \end{equation} 
  \end{enumerate} 
  \end{theorem}
  
  This theorem combined with the Theorem 2.3 of \cite{hule} gives
    \begin{corollary}\label{cor:uhder} Assuming that the hypothesis \ref{h.2.1}   and $u_0$ satisfies \eqref{cond:fu0}. 
  Let $\bar  \alpha_0$ and $\bar\alpha$ be in $[0,1]$ such that
  \begin{equation}
  \max( 2\bar \alpha_0, \bar \alpha) <
  \begin{cases}
  2- \alpha_0- \frac\alpha2 &\qquad \hbox{if hypothesis \ref{h.2.2} (i) be  satisfied}\\
  \frac{3-2\al_0-\al}{2} &\qquad \hbox{if hypothesis \ref{h.2.2} (ii) be  satisfied}\\
  \end{cases}  
  \end{equation}
\begin{enumerate}
\item[(i)]   For any $M>0$  there is a random constant $C_M$ such that   
   every $t,r\in[0,M]$ and for every $x,y\in\RR^d$ satisfying $|x|,|y|\le M$,
  \begin{equation}
  |u (t,y)-u (r,y)-u (t,x)+u (r,x)|\le C |t-r|^{ \bar\al_0} |x-y|^{ \bar\al}\,.
  \end{equation}
\item[(ii)]   For any $M>0$  there is a random constant $C_M$ such that   
   every $t,r\in[0,M]$ and for every $x,y\in\RR^d$ satisfying $|x|,|y|\le M$,
  \begin{equation}
  | u (t,x)-u (r,x)|+ 
  |u (t,y) -u (t,x) |\le V\left[ |t-s|^{\bar \al_0} +|x-y|^{\bar \alpha}\right] \,.
  \label{e.2.8}
  \end{equation} 
 \end{enumerate} 
  \end{corollary}
  \begin{remark}\label{r.5.5} 
   When $d=1$ and $W$ is  a space-time white noise, this corresponds to the case $\alpha_0=\alpha=1$ ans we can use the result for the case when the hypothesis \ref{h.2.2} (i) is satisfied.  The above corollary  says
  if $2\bar\al_0+\bar \alpha <1/2$,  then
  \begin{equation}
  |u(t,y)-u(r,y)-u(t,x)+u(r,x)|
  \le  C |t-r|^{\bar \al_0} |x-y|^{\bar \alpha}\,.
  \end{equation}
  This coincides with the optimal H\"older exponent result in \cite{hule}.
  On the other hand,  the corollary also implies  in this case  that  if $\bar\al_0<1/4$ and $\bar \alpha<1/2$,  then  
  \begin{equation}
  |u(t,x)-u(r,x)|+|u(t,y)-u(t,x)|
  \le  C \left[|t-r|^{\bar \al_0}+ |x-y|^{ \bar \alpha}\right]\,. 
  \end{equation} 
  From here we   see that $u(t,x)$ is H\"older continuous
  in $t$ with exponent $\bar\al_0<1/4$ and is H\"older continuous
  in $x$ with exponent $\bar\al<1/2$.  This is the optimal H\"older modulus of continuity for the solution of \eqref{spde} with one dimensional space-time white noise. 
  \end{remark}

\end{document}